\documentclass{amsart}
\usepackage{amsmath}
\usepackage{amssymb}
\usepackage{amscd}
\usepackage{mathrsfs}

\newtheorem{theorem}{Theorem}[section]
\newtheorem{proposition}[theorem]{Proposition}
\newtheorem{lemma}[theorem]{Lemma}

\theoremstyle{definition}
\newtheorem{definition}[theorem]{Definition}
\newtheorem{theorem-definition}{Theorem-Definition}[theorem]
\newtheorem{notation}[theorem]{Notation}
\newtheorem{example}[theorem]{Example}

\theoremstyle{remark}
\newtheorem{remark}[theorem]{Remark}
\newtheorem*{acknowledgments}{Acknowledgments}

\numberwithin{equation}{section}

\newcommand{\Z}{{\mathbb{Z}}}
\newcommand{\Q}{{\mathbb{Q}}}
\newcommand{\N}{{\mathbb{N}}}
\newcommand{\T}{{\mathbb{T}}}
\newcommand{\R}{{\mathbb{R}}}
\newcommand{\C}{{\mathbb{C}}}
\newcommand{\Qadic}[1]{{\Z\left[\frac{1}{#1}\right]}}
\newcommand{\Qadicsq}[1]{{\Qadic{#1}\times\Qadic{#1}}}

\newcommand{\solenoid}{{\mathscr{S}}}
\newcommand{\algebra}{\mathscr{A}^{\solenoid}}
\newcommand{\bigslant}[2]{{\raisebox{.2em}{$#1$}\left/\raisebox{-.2em}{$#2$}\right.}}
\newcommand{\dual}[1]{{\widehat{#1}}}

\begin{document}

\title{Noncommutative solenoids and their projective modules}

\author{Fr\'{e}d\'{e}ric Latr\'{e}moli\`{e}re}
\address{Department of Mathematics, University of Denver, 80208}
\email{frederic@math.du.edu}
\author{Judith A. Packer}
\address{Department of Mathematics, Campus Box 395, University of Colorado, Boulder, CO, 80309-0395}
\email{packer@euclid.colorado.edu}
\date{March 31, 2013}


\begin{abstract}
Let $p$ be prime. A noncommutative $p$-solenoid is the $C^{\ast}$-algebra of $\Qadicsq{p}$ twisted by a multiplier of that group, where $\Qadic{p}$ is the additive subgroup of the field $\mathbb{Q}$ of rational numbers whose denominators are powers of $p$. In this paper, we survey our classification of these $C^{\ast}$-algebras up to *-isomorphism in terms of the multipliers on $\Qadicsq{p}$, using  techniques from noncommutative topology. Our work relies in part on writing these $C^{\ast}$-algebras as direct limits of rotation algebras, i.e. twisted group C*-algebras of the group $\Z^2,$ thereby providing a mean for computing the $K$-theory of the noncommutative solenoids, as well as the range of the trace on the $K_0$ groups. We also establish a necessary and sufficient condition for the simplicity of the noncommutative solenoids.  Then, using the computation of the trace on $K_0$, we discuss two different ways of constructing projective modules over the noncommutative solenoids.

\end{abstract}

\maketitle


\section{Introduction}
Twisted group algebras and transformation group $C^{\ast}$-algebras have been studied since the early 1960's \cite{Kleppner65} and provide a rich source of examples and problems in C*-algebra theory. Much progress has been made in studying such $C^{\ast}$-algebras when the groups involved are finitely generated (or compactly generated, in the case of Lie groups).  Even when $G=\mathbb Z^n,$ these $C^{\ast}$-algebras give a rich class of examples which have driven much development in C*-algebra theory, including the foundation of noncommutative geometry by Connes \cite{Connes80}, the extensive study of the geometry of quantum tori by Rieffel \cite{Rieffel81,Rieffel88,Rieffel90,Rieffel98a}, the expansion of the classification problem from AF to AT algebras by G. Elliott and D. Evans \cite{Elliott93b}, and many more (L. Baggett and A. Kleppner \cite{BK76}, and S. Echterhoff and J. Rosenberg \cite{ER95}).

In this paper, we present our work on twisted group $C^\ast$-algebras of the Cartesian square of the discrete group $\Qadic{p}$ of $p$-adic rationals, i.e. the additive subgroup of $\Q$ whose elements have denominators given by powers of a fixed $p \in \N,p\geq 1$. The Pontryagin duals of these groups are the $p$-solenoid, thereby motivating our terminology in calling these $C^{\ast}$-algebras 
 {\it noncommutative solenoids}.  We review our computation of the $K$-groups of these $C^{\ast}$-algebras, derived in their full technicality in \cite{Latremoliere11c}, and which in and of itself involves an intriguing problem in the theory of Abelian group extensions. We were also able to compute the range of the trace on the $K_0$-groups, and use this knowledge to classify these $C^{\ast}$-algebras up to $\ast$-isomorphism, in \cite{Latremoliere11c}, and these facts are summarized in a brief survey of \cite{Latremoliere11c} in the first two sections of this paper.  

This paper is concerns with the open problem of classifying noncommutative solenoids up to Morita equivalence. We demonstrate a method of constructing an equivalence bimodule between two noncommutative solenoids using methods due to M. Rieffel \cite{Rieffel88}, and will note how this method has relationships to the theory of wavelet frames. These new matters occupy the last three sections of this paper.

\begin{acknowledgments}
The authors gratefully acknowledge helpful conversations with  
Jerry Kaminker and Jack Spielberg.
\end{acknowledgments}


\section{Noncommutative Solenoids}

This section and the next provide a survey of the main results proven in \cite{Latremoliere11c} concerning the computation of the $K$-theory of noncommutative solenoids and its application to their classification up to *-isomorphism. An interesting connection between the $K$-theory of noncommutative solenoids and the $p$-adic integers is unearthed, and in particular, we prove that the range of the $K_0$ functor on the class of all noncommutative solenoids is fully described by all Abelian extensions of the group of $p$-adic rationals by $\Z$. These interesting matters are the subject of the next section, whereas we start in this section with the basic objects of our study.

We shall fix, for this section and the next, an arbitrary $p \in \N$ with $p > 1$. Our story starts with the following groups:

\begin{definition}
Let $p\in\N,p>1$. The group $\Qadic{p}$ of $p$-adic rationals is the inductive limit of the sequence of groups:
\begin{equation*}
\begin{CD}
\Z @>z\mapsto p z>> \Z @>z\mapsto p z>> \Z @>z\mapsto p z>> \Z @>z\mapsto p z>> \cdots
\end{CD}
\end{equation*}
which is explicitly given as the group:
\begin{equation}\label{NadicRationals}
\Qadic{p} = \left\{ \frac{z}{p^k} \in \Q : z\in\Z, k\in\N \right\}
\end{equation}
endowed with the discrete topology.
\end{definition}
From the description of $\Qadic{p}$ as an injective limit, we obtain the following result by functoriality of the Pontryagin duality. We denote by $\T$ the unit circle $\{ z \in \C : |z|=1\}$ in the field $\C$ of complex numbers.
\begin{proposition}\label{solenoids}
Let $p\in\N,p>1$. The Pontryagin dual of the group $\Qadic{p}$ is the $p$-solenoid group, given by:
\[
\solenoid_p = \left\{ (z_n)_{n\in\N} \in \T^{\N}: \forall n \in \N \;\; z_{n+1}^p = z_n \right\}\text{,}
\]
endowed with the induced topology from the injection $\solenoid_p\hookrightarrow \T^{\N}$. The dual pairing between $\Q_N$ and $\solenoid_N$ is given by:
\[
\left< \frac{q}{p^k}, (z_n)_{n\in\N} \right> = z_k^q \text{,}
\]
where $\frac{q}{p^k} \in \Qadic{p}$ and $(z_n)_{n\in\N}\in\solenoid_p$.
\end{proposition}

We study in \cite{Latremoliere11c} the following C*-algebras.
\begin{definition}\label{ncsolenoid-def}
A \emph{noncommutative solenoid} is a C*-algebra of the form
\begin{equation*}
C^\ast\left(\Qadicsq{p},\sigma\right)\text{,}
\end{equation*}
where $p$ is a natural number greater or equal to $2$ and $\sigma$ is a multiplier of the group $\Qadicsq{p}$.
\end{definition}

The first matter to attend in the study of these C*-algebras is to describe all the multipliers of the group $\Qadicsq{p}$ up to equivalence, where our multipliers are $\T$-valued unless otherwise specified, with $\T$ the unit circle in $\C$. Note that the group $\Qadic{p}$ has no nontrivial multiplier, so our noncommutative solenoids are the natural object to consider.

Using \cite{Kleppner65}, we compute in \cite{Latremoliere11c} the group $H^2\left(\Qadicsq{p},\T\right)$ of $\T$-valued multipliers of $\Qadicsq{p}$ up to equivalence, as follows:

\begin{theorem}\cite[Theorem 2.3]{Latremoliere11c}\label{multipliers}
Let $p\in\N,p>1$. Let:
\begin{equation*}
\Xi_p = \left\{ \left( \alpha_n \right) : 
\alpha_0 \in [0,1) \,\,\wedge\,\, (\forall n\in \N \,\, \exists k \in \{0,\ldots,N-1\}\;\; p \alpha_{n+1} =  \alpha_n+k )
\right\}
\end{equation*}
which is a group for the pointwise addition modulo one. There exists a group isomorphism $\rho : H^2\left(\Qadicsq{p},\T\right) \rightarrow \Xi_p$ such that if $\sigma\in H^2\left(\Qadicsq{p},\T\right)$ and $\alpha=\rho(\sigma)$, and if $f$ is a multiplier of class $\sigma$, then $f$ is cohomologous to:
\begin{equation*}
\Psi_\alpha:\left(\left(\frac{q_1}{p^{k_1}},\frac{q_2}{p^{k_2}}\right),\left(\frac{q_3}{p^{k_3}},\frac{q_4}{p^{k_4}}\right)\right) \longmapsto \exp\left(2i\pi \alpha_{(k_1+k_4)}q_1q_4\right)\text{.}
\end{equation*}
\end{theorem}
For any $p \in \N, p > 1$, the groups $\Xi_p$ and $\solenoid_p$ are obviously isomorphic as topological groups; yet it is easier to perform our computations in the additive group $\Xi_p$ in what follows. Thus, as a topological group, $H^2\left(\Qadicsq{p},\T\right)$ is isomorphic to $\solenoid_p$. Moreover, we observe that a corollary of Theorem (\ref{multipliers}) is that $\Psi_\alpha$ and $\Psi_\beta$ are cohomologous if and only if $\alpha = \beta \in \Xi_p$.  The proof of Theorem (\ref{multipliers}) involves the standard calculations for cohomology classes of multipliers on discrete Abelian groups, due to A. Kleppner, generalizing results of Backhouse and Bradley.

With this understanding of the multipliers of $\Qadicsq{p}$, we thus propose to classify the noncommutative solenoids $C^\ast\left(\Qadicsq{p},\sigma\right)$.  Let us start by recalling \cite{Zeller-Meier68} that for any multiplier $\sigma$ of a discrete group $\Gamma$, the C*-algebra $C^{\ast}\left(\Gamma,\sigma\right)$ is the $C^{\ast}$-completion of the involutive Banach algebra $\left(\ell^1\left(\Gamma\right),\ast_\sigma,\cdot^\ast\right)$, where the twisted convolution $\ast_\sigma$ is given for any $f_1,f_2\in\ell^1\left(\Gamma\right)$ by 
\begin{equation*}
f_1 \ast_\sigma f_2: \gamma\in\Gamma\longmapsto\sum_{\gamma_1\in \Gamma}f_1(\gamma_1)f_2(\gamma-\gamma_1)\sigma(\gamma_1,\gamma-\gamma_1)\text{,}
\end{equation*}
while the adjoint operation is given by:
\begin{equation*}
f_1^{\ast}:\gamma\in\Gamma \longmapsto \overline{\sigma(\gamma,-\gamma)f_1(-\gamma)}\text{.}
\end{equation*}
The C*-algebra $C^\ast\left(\Gamma,\sigma\right)$ is then shown to be the universal C*-algebra generated by a family $(W_\gamma)_{\gamma\in\Gamma}$ of unitaries such that $W_\gamma W_\delta = \sigma(\gamma,\delta)W_{\gamma\delta}$ for any $\gamma,\delta\in \Gamma$ \cite{Zeller-Meier68}. We shall henceforth refer to these generating unitaries as the \emph{canonical unitaries} of $C^\ast\left(\Gamma,\sigma\right)$.

One checks easily that if $\sigma$ and $\eta$ are two cohomologous multipliers of the discrete group $\Gamma$, then $C^\ast\left(\Gamma,\sigma\right)$ and $C^\ast\left(\Gamma,\eta\right)$ are *-isomorphic \cite{Zeller-Meier68}. Thus, by Theorem (\ref{multipliers}), we shall henceforth restrict our attention to multipliers of $\Qadicsq{p}$ of the form $\Psi_\alpha$ with $\alpha\in\Xi_p$. With this in mind, we introduce the following notation:
\begin{notation}\label{solenoid-not}
For any $p\in\N, p > 1$ and for any $\alpha \in \Xi_p$, the C*-algebra $C^\ast\left(\Qadicsq{p},\Psi_\alpha\right)$, with $\Psi_\alpha$ defined in Theorem (\ref{multipliers}), is denoted by $\algebra_\alpha$.
\end{notation}

Noncommutative solenoids, defined in Definition (\ref{ncsolenoid-def}) as twisted group algebras of $\Qadicsq{p}$, also have a presentation as transformation group $C^{\ast}$-algebras, in a manner similar to the situation with rotation C*-algebras:

\begin{proposition}\cite[Proposition 3.3]{Latremoliere11c}\label{crossed-product}
Let $p\in\N,p>1$ and $\alpha\in\Xi_p$. Let $\theta^\alpha$ be the action of $\Qadic{p}$ on $\solenoid_p$ defined for all $\frac{q}{p^k} \in \Qadic{p}$ and for all $(z_n)_{n\in\N} \in \solenoid_p$ by:
\[
\theta^\alpha_{\frac{q}{p^k}} \left((z_n)_{n\in\N} \right) = \left(e^{\left(2i\pi\alpha_{(k+n)} q\right)} z_n \right)_{n\in\N}\text{.}
\]
The C*-crossed-product $C(\solenoid_p)\rtimes_{\theta^\alpha} \Qadic{p}$ is *-isomorphic to $\algebra_\alpha$.
\end{proposition}

Whichever way one decides to study them, there are longstanding methods in place to determine whether or not these $C^{\ast}$ algebras are simple (see for instance \cite{}).  For now, we concentrate on methods from the theory of twisted group $C^{\ast}$-algebras. 

\begin{theorem-definition}\cite[Theorem 1.5]{Packer92}\label{simplicity-thm}
The \emph{symmetrizer group} of a multiplier $\sigma: \Gamma\times\Gamma\rightarrow\T$ of a discrete group $\Gamma$ is given by
\begin{equation*}
S_{\sigma}=\left\{\gamma\in \Gamma: \forall g \in \Gamma \quad \sigma(\gamma, g)\sigma(g,\gamma)^{-1}=1\right\}\text{.}
\end{equation*}
The C*-algebra $C^{\ast}(\Gamma,\sigma)$ is simple if, and only if the symmetrizer group $S_\sigma$ is reduced to the identity of $\Gamma$.
\end{theorem-definition}

In \cite{Latremoliere11c}, we thus characterize when the symmetrizer group of the multipliers of $\Qadicsq{p}$ given by Theorem (\ref{multipliers}) is non-trivial:

\begin{theorem}\cite[Theorem 2.12]{Latremoliere11c}\label{symmetrizer-computation-thm}
Let $p\in\N,p>1$. Let $\alpha \in \Xi_p$. Denote by $\Psi_\alpha$ the multiplier defined in Theorem (\ref{multipliers}). The following are equivalent:
\begin{enumerate}
  \item the symmetrizer group $S_{\Psi_\alpha}$ is non-trivial,
  \item the sequence $\alpha$ has finite range, i.e. the set $\{\alpha_j : j \in \N\}$ is finite,
  \item there exists $k\in \N$ such that $(p^k-1)\alpha_0\in \Z$, 
  \item the sequence $\alpha$ is periodic,
  \item there exists a positive integer $b\in\N$ such that:
\begin{equation*}
S_{\Psi_\alpha}= b\Qadicsq{p}=\left\{ (br_1,br_2),\;(r_1,r_2)\in \Qadicsq{p}\right\}\text{.}
\end{equation*}
\end{enumerate}
\end{theorem}

Theorem (\ref{simplicity-thm}), when applied to noncommutative solenoids via Theorem (\ref{symmetrizer-computation-thm}), allows us to conclude:

\begin{theorem}\cite[Theorem 3.5]{Latremoliere11c}\label{simplicity-thm}
Let $p\in\N,p>1$ and $\alpha \in \Xi_p$. Then the following are equivalent:
\begin{enumerate}
  \item the noncommutative solenoid $\algebra_\alpha$ is simple,
  \item the set $\{\alpha_j : j \in \N \}$ is infinite,
  \item for every $k\in\N$, we have $(p^k-1)\alpha_0\not\in \Z$. 
\end{enumerate}
\end{theorem}


In particular, if $\alpha\in\Xi_p$ is chosen with at least one irrational entry, then by definition of $\Xi_p$, all entries of $\alpha$ are irrational, and by Theorem (\ref{simplicity-thm}), the noncommutative solenoid $\algebra_\alpha$ is simple. The reader may observe that, even if $\alpha \in \Xi_p$ only has rational entries, the noncommutative solenoid may yet be simple --- as long as $\alpha$ has infinite range. We called this situation the {\it aperiodic rational case} in \cite{Latremoliere11c}. 

\begin{example}[Aperiodic rational case]
Let $p=7$, and consider $\alpha\in \Xi_7$ given by 
\begin{equation*}
\alpha = \left(\frac{2}{7}, \frac{2}{49}, \frac{2}{343}, \frac{2}{2401},\cdots\right) = \left( \frac{2}{7^n}\right)_{n\in\N}\text{.}
\end{equation*}
Note that $\alpha_j\in\mathbb Q$ for all $j\in\N$, yet Theorem (\ref{simplicity-thm}) tells us that the noncommutative solenoid $\algebra_\alpha$ is simple!
\end{example}

The following is an example where the symmetrizer subgroup is non-trivial, so that the corresponding $C^{\ast}$-algebra is not simple.

\begin{example}[Periodic rational case]
Let $p=5$, and consider $\alpha\in \Xi_5$ given by 
\begin{equation*}
\alpha = \left(\frac{1}{62}, \frac{25}{62}, \frac{5}{62}, \frac{1}{62},\cdots \right)\text{.}
\end{equation*}

Theorem (\ref{symmetrizer-computation-thm}) shows that the symmetrizer group of the multiplier $\Psi_\alpha$ of $\left(\Qadic{5}\right)^2$ given by Theorem (\ref{multipliers}) is:
\begin{equation*}
S_{\alpha}=\left\{\left(\frac{62 j_1}{5^k}, \frac{62 j_2}{5^k}\right) \in \Q:j_1, j_2 \in \Z, k\in \N\right\}\text{.}
\end{equation*}
Hence the noncommutative solenoid $\algebra_\alpha$ is not simple by Theorem (\ref{simplicity-thm}).
\end{example}

We conclude this section with the following result about the existence of traces on noncommutative solenoids, which follows from \cite{Hoegh-Krohn81}, since the Pontryagin dual $\solenoid_p\times\solenoid_p$ of $\Qadicsq{p}$ acts ergodically on $\algebra_\alpha$ for any $\alpha\in\Xi_p$ via the dual action:

\begin{theorem}\cite[Theorem 3.8]{Latremoliere11c}
Let $p\in\N,p>1$ and $\alpha\in\Xi_p$. There exists at least one tracial state on the noncommutative solenoid $\algebra_\alpha$. Moreover, this tracial state is unique if, and only if $\alpha$ is not periodic.
\end{theorem}

Moreover, since noncommutative solenoids carry an ergodic action of the compact groups $\solenoid_p$, if one chooses any continuous length function on $\solenoid_p$, then one may employ the results found in \cite{Rieffel98a} to equip noncommutative solenoids with quantum compact metric spaces structures and, for instance, use \cite{Rieffel00} and \cite{Latremoliere05} to obtain various results on continuity for the quantum Gromov-Hausdorff distance of the family of noncommutative solenoids as the multiplier and the length functions are left to vary. In this paper, we shall focus our attention on the noncommutative topology of our noncommutative solenoids, rather than their metric properties.

In \cite[Theorem 3.17]{Latremoliere11c}, we provide a full description of noncommutative solenoids as bundles of matrix algebras over the space $\solenoid_p^2$, while in contrast, in \cite[Proposition 3.16]{Latremoliere11c}, we note that for $\alpha$ with at least (and thus all) irrational entry, the noncommutative solenoid $\algebra_\alpha$ is an inductive limit of circle algebras (i.e. AT), with real rank zero. Both these results follow from writing noncommutative solenoids as inductive limits of quantum tori, which is the starting point for our next section.

\section{Classification of the noncommutative Solenoids} 

Noncommutative solenoids are classified by their $K$-theory; more precisely by their $K_0$ groups and the range of the traces on $K_0$. The main content of our paper \cite{Latremoliere11c} is the computation of the $K$-theory of noncommutative solenoids and its application to their classification up to *-isomorphism.

The starting point of this computation is the identification of noncommutative solenoids as inductive limits of sequences of noncommutative tori. A \emph{noncommutative torus} is a twisted group C*-algebra for $\Z^d$, with $d\in \N,d>1$ \cite{Rieffel81}. In particular, for $d=2$, we have the following description of noncommutative tori. Any multiplier of $\Z^2$ is cohomologous to one of the form:
\begin{equation*}
\sigma_\theta : \left(\begin{pmatrix}z_1\\z_2\end{pmatrix},\begin{pmatrix}y_1\\y_2\end{pmatrix}\right)\longmapsto \exp(2i\pi\theta z_1y_2)
\end{equation*}
for some $\theta\in [0,1)$. Consequently, for a given $\theta \in [0,1)$, the C*-algebra $C^\ast\left(\Z^2,\sigma_\theta\right)$ is the universal C*-algebra generated by two unitaries $U,V$ such that:
\begin{equation*}
UV = e^{2i\pi\theta}VU\text{.}
\end{equation*}
We will employ the following notation throughout this paper:
\begin{notation}\label{qtorus-not}
The noncommutative torus $C^\ast\left(\Z^2,\sigma_\theta\right)$, for $\theta\in[0,1)$, is denoted by $A_\theta$. Moreover, the two canonical generators of $A_\theta$ (i.e. the unitaries corresponding to $(1,0),(0,1) \in \Z^2$), are denoted by $U_\theta$ and $V_\theta$, so that $U_\theta V_\theta = e^{2i\pi\theta}V_\theta U_\theta$.
\end{notation}

For any $\theta\in[0,1)$, the noncommutative torus $A_\theta$ is *-isomorphic to the crossed-product C*-algebra for the action of $\Z$ on the circle $\T$ generated by the rotation of angle $2i\pi\theta$, and thus $A_\theta$ is also known as the rotation algebra for the rotation of angle $\theta$ --- a name by which it was originally known.

The following question naturally arises: since $A_{\theta}$ is a twisted $\Z^2$ algebra, and $\Qadicsq{p}$ can be realized as a direct limit group built from embeddings of $\Z^2$ into itself, is it possible  to  build our noncommutative solenoids $\algebra_\alpha$ as a direct limits of rotation algebras? The answer is positive, and this observation provides much structural information regarding noncommutative solenoids.

\begin{theorem}\cite[Theorem 3.7]{Latremoliere11c}\label{inductive-limit-thm}
Let $p\in\N,p>1$ and $\alpha\in\Xi_p$. For all $n\in\N$, let $\varphi_n$ be the unique *-morphism from $A_{\alpha_{2n}}$ into $A_{\alpha_{2n+2}}$ given by:
\begin{equation*}
\begin{cases}
U_{\alpha_{2n}} \longmapsto U_{\alpha_{2n+2}}^p\\
V_{\alpha_{2n}} \longmapsto V_{\alpha_{2n+2}}^p\\
\end{cases}
\end{equation*}
Then:
\[
A_{\alpha_0}\; \stackrel{\varphi_0}{\longrightarrow}\; A_{\alpha_2}\;\stackrel{\varphi_1}{\longrightarrow}\;A_{\alpha_4} \;\stackrel{\varphi_2}{\longrightarrow}\; \cdots
\]
converges to the noncommutative solenoid $\algebra_\alpha$. Moreover, if $(W_{r_1,r_2})_{ (r_1,r_2)\in\Qadicsq{p}}$ is the family of canonical unitary generators of $\algebra_\alpha$, then, for all $n\in\N$, the rotation algebra $A_{\alpha_{2n}}$ embeds in $\algebra_\alpha$ via the unique extension of the map:
\begin{equation*}
\begin{cases}
U_{\alpha_{2n}} \longmapsto W_{\left(\frac{1}{p^n},0\right)}\\
V_{\alpha_{2n}} \longmapsto W_{\left(0,\frac{1}{p^n}\right)} \text{.}
\end{cases}
\end{equation*}
to a *-morphism, given by the universal property of rotation algebras; one checks that this embeddings, indeed, commute with the maps $\varphi_n$. 
\end{theorem}

Our choice of terminology for noncommutative solenoids is inspired, in part, by Theorem (\ref{inductive-limit-thm}), and the well established terminology of noncommutative torus for rotation algebras. Moreover, as we shall now see, our study of noncommutative solenoids is firmly set within the framework of noncommutative topology.

The main result from our paper \cite{Latremoliere11c} under survey in this section and the previous one is the computation of the $K$-theory of noncommutative solenoid and its application to their classification. An interesting connection between our work on noncommutative solenoid and classifications of Abelian extensions of $\Qadic{p}$ by $\Z$, which in turn are classified by means of the group of $p$-adic integers, emerges as a consequence of our computation. We shall present this result now, starting with some reminders about the $p$-adic integers and Abelian extensions of $\Qadic{p}$, and refer to \cite{Latremoliere11c} for the involved proof leading to it.

\begin{theorem-definition}
Let $p\in\N,p>1$. The set:
\begin{equation*}
\Z_p = \left\{ (J_k)_{k\in\N} : J_0 = 0 \text{ and }\forall k \in \N\quad J_{k+1} \equiv J_k \mod p^k \right\}
\end{equation*}
is a group for the operation defined as:
\begin{equation*}
(J_k)_{k\in\N} + (K_k)_{K\in\N} = ( (J_k + K_k) \mod p^k)_{k\in\N}
\end{equation*}
for any $(J_k)_{k\in\N},(K_k)_{k\in\N} \in \Z_p$. This group is \emph{the group of $p$-adic integers.}
\end{theorem-definition}

One may define the group of $p$-adic integer simply as the set of sequences valued in $\{0,\ldots,p-1\}$ with the appropriate operation, but our choice of definition will make our exposition clearer. We note that we have a natural embedding of $\Z$ as a subgroup of $\Z_p$ by sending $z\in\Z$ to the sequence $(z \mod p^k)_{k\in\N}$. We shall henceforth identify $\Z$ with its image in $\Z_p$ when no confusion may arise.

We can associate, to any $p$-adic integer, a Schur multiplier of $\Qadic{p}$, i.e. a map $\xi_j : \Qadicsq{p} \rightarrow\Z$ which satisfies the (additive) $2$-cocycle identity, in the following manner:

\begin{theorem}\label{schur-thm}\cite{Latremoliere11c}
Let $p\in\N, p >1$ and let $J = (J_k)_{k\in\N} \in \Z_p$. Define the map $\xi_J : \Qadicsq{p} \rightarrow\Z$ by setting, for any $\frac{q_1}{p^{k_1}},\frac{q_2}{p^{k_2}} \in \Qadic{p}$:
\begin{equation*}
\xi_J\left(\frac{q_1}{p^{k_1}},\frac{q_2}{p^{k_2}}\right) = \begin{cases}
-\frac{q_1}{p^{k_1}}\left(J_{k_2}-J_{k_1}\right) &\text{ if $k_2 > k_1$,}\\
-\frac{q_2}{p^{k_2}}\left(J_{k_1}-J_{k_2}\right) &\text{ if $k_1 > k_2$,}\\
\frac{q}{p^r}\left(J_{k_1}-J_r\right) &\text{ if $k_1=k_2$, with $\frac{q}{p^r} = \frac{q_1}{p^{k_1}}+\frac{q_2}{p^{k_2}}$,}
\end{cases}
\end{equation*}
where \emph{all} fractions are written in their \emph{reduced form}, i.e. such that the exponent of $p$ at the denominator is minimal (this form is unique). Then:
\begin{itemize}
\item $\xi_J$ is a Schur multiplier of $\Qadic{p}$ \cite[Lemma 3.11]{Latremoliere11c}.
\item For any $J,K \in \Z_p$, the Schur multipliers $\xi_J$ and $\xi_K$ are cohomologous if, and only if $J-K \in \Z$ \cite[Theorem 3.14]{Latremoliere11c}.
\item Any Schur multiplier of $\Qadic{p}$ is cohomologous to $\xi_J$ for some $J\in\Z_p$ \cite[Theorem 3.16]{Latremoliere11c}.
\end{itemize}
In particular, $\mathrm{Ext}\left(\Qadic{p},\Z\right)$ is isomorphic to $\bigslant{\Z_p}{\Z}$.
\end{theorem}

Schur multipliers provide us with a mean to describe and classify Abelian extensions of $\Qadic{p}$ by $\Z_p$. Our interest in Theorem (\ref{schur-thm}) lies in the remarkable observation that the $K_0$ groups of noncommutative solenoids are exactly given by these extensions:

\begin{theorem}\cite[Theorem 3.12]{Latremoliere11c}\label{K-theory-thm}
Let $p\in\N,p>1$ and let $\alpha = (\alpha_k)_{k\in\N} \in \Xi_p$. For any $k\in \N$, define $J_k = p^k \alpha_k - \alpha_0$, and note that by construction, $J \in \Z_p$. Let $\xi_J$ be the Schur multiplier of $\Qadic{p}$ defined in Theorem (\ref{schur-thm}), and let $\mathscr{Q}_J$ be the group with underlying set $\Z\times \Qadic{p}$ and operation:
\begin{equation*}
\left(z_1,r_1\right)\boxplus\left(z_2,r_2\right) = \left(z_1 + z_2 + \xi_J\left(r_1,r_2\right),r_1+r_2\right)
\end{equation*}
for all $(z_1,r_1),(z_2,r_2)\in\Z\times\Qadic{p}$. By construction, $\mathscr{Q}_J$ is an Abelian extension of $\Qadic{p}$ by $\Z$ given by the Schur multiplier $\xi_J$. 

Then:
\begin{equation*}
K_0 \left( \algebra_\alpha \right) = \mathscr{Q}_J
\end{equation*}
and, moreover, all tracial states of $\algebra_\alpha$ lift to a single trace $\tau$ on $K_0\left(\algebra_\alpha\right)$, characterized by:
\begin{equation*}
\tau: (1,0)\mapsto 1 \text{ and }\left(0,\frac{1}{p^k}\right)\mapsto \alpha_k\text{.}
\end{equation*}
Furthermore, we have:
\begin{equation*}
K_1\left(\algebra_\alpha\right) = \Qadicsq{p}\text{.}
\end{equation*}
\end{theorem}

We observe, in particular, that given any Abelian extension of $\Qadic{p}$ by $\Z$, one can find, by Theorem (\ref{schur-thm}), a Schur multiplier of $\Qadic{p}$ of the form $\xi_J$ for some $J\in \Z_p$, and, up to an arbitrary choice of $\alpha_0 \in [0,1)$, one may form the sequence $\alpha = \left(\frac{\alpha_0+J_k}{p^k}\right)_{k\in\N}$, and check that $\alpha \in \Xi_p$; thus \emph{all} possible Abelian extensions, and \emph{only} Abelian extensions of $\Qadic{p}$ by $\Z$ are given as $K_0$ groups of noncommutative solenoids. With this observation, the $K_0$ groups of noncommutative solenoids are uniquely described by a $p$-adic integer modulo an integer, and the information contained in the pair $(K_0(\algebra_\alpha),\tau)$ of the $K_0$ group of a noncommutative solenoid and its trace, is contained in the pair $(J,\alpha_0)$ with $J\in\bigslant{\Z_p}{\Z}$ as defined in Theorem (\ref{K-theory-thm}). 

\begin{remark}\label{trace-range-rmk}
For any $p\in\N,p>1$ and $\alpha\in\Xi_p$, the range of the unique trace $\tau$ on $K_0(\algebra_\alpha)$, as described by Theorem (\ref{K-theory-thm}), is the subgroup $\Z\oplus \oplus_{k\in\N}\alpha_n\Z$. Let $\gamma = z+z_1\alpha_{n_1}+\ldots z_k\alpha_{n_k}$ be an arbitrary element of this set, where, to fix notations, we assume $n_1<\ldots<n_k$. Then, since $\alpha_{n+1} \equiv p \alpha_n \mod 1$ for any $n\in\N$, we conclude that we can rewrite $\gamma$ simply as $z'+y\alpha_{n_k}$, for some $z',y \in \Z$. Thus the range of our trace on $K_0(\algebra_\alpha)$ is given by:
\begin{equation*}
\tau\left(K_0\left(\algebra_\alpha\right)\right) = \left\{ z+ y \alpha_k : z,y \in \Z, k \in \N \right\}\text{.}
\end{equation*}
\end{remark}

We thus have a complete characterization of the $K$-theory of noncommutative solenoids. This noncommutative topological invariant, in turn, contains enough information to fully classify noncommutative solenoids in term of their defining multipliers. We refer to \cite[Theorem 4.2]{Latremoliere11c} for the complete statement of this classification; to keep our notations at a minimum, we shall state the corollary of \cite[Theorem 4.2]{Latremoliere11c} when working with $p$ prime:

\begin{theorem}\label{classification-thm}\cite[Corollary 4.3]{Latremoliere11c}
Let $p,q$ be two prime numbers and let $\alpha \in \Xi_p$ and $\beta\in\Xi_q$. Then the following are equivalent:
\begin{enumerate}
\item The noncommutative solenoids $\algebra_\alpha$ and $\algebra_\beta$ are *-isomorphic,
\item $p = q$ and a truncated subsequence of $\alpha$ is a truncated subsequence of $\beta$ or $(1-\beta_k)_{k\in\N}$.
\end{enumerate}
\end{theorem}

Theorem (\ref{classification-thm}) is given in greater generality in \cite[Theorem 4.2]{Latremoliere11c}, where $p,q$ are not assumed prime; the second assertion of the Theorem must however be phrased in a more convoluted manner: essentially, $p$ and $q$ must have the same set of prime factors, and there is an embedding of $\Xi_p$ and $\Xi_q$ in a larger group $\Xi$, whose elements are still sequences in $[0,1)$, such that the images of $\alpha$ and $\beta$ for these embeddings are sub-sequences of a single element of $\Xi$. 

We conclude this section with an element of the computation of the $K_0$ groups in Theorem (\ref{K-theory-thm}). Given $\gamma = \left(0,\frac{1}{p^k}\right) \in K_0\left(\algebra_\alpha\right)$, if $\alpha_0$ is irrational, then there exists a Rieffel-Powers projection in $A_{\alpha_{2k}}$ whose image in $\algebra_\alpha$ for the embedding given by Theorem (\ref{inductive-limit-thm}) has $K_0$ class the element $\gamma$, whose trace is thus naturally given by Theorem (\ref{K-theory-thm}). Much work is needed, however, to identify the range of $K_0$ as the set of all Abelian extensions of $\Qadic{p}$ by $\Z$, and parametrize these, in turn, by $\bigslant{\Z_p}{\Z}$, as we have shown in this section.

We now turn to the question of the structure of the category of modules over noncommutative solenoids. In the next two sections, we show how to apply some constructions of equivalence bimodules to the case of noncommutative solenoids as a first step toward solving the still open problem of Morita equivalence for noncommutative solenoids.

\section{Forming projective modules over noncommutative solenoids from the inside out }

Projective modules for rotation algebras and higher dimensional noncommutative tori were studied by M. Rieffel (\cite{Rieffel88}). F. Luef has extended this work to build modules with a dense subspace of functions coming from modulation spaces (e.g., Feichtinger's algebra) with nice properties (\cite{Luef09}, \cite{Luef11}).   One approach to building projective modules over noncommutative solenoids is to build the projective modules from the ``inside out".

We first make some straightforward observations in this direction. We recall that, by Notation (\ref{solenoid-not}), for any $p\in\N,p>1$, and for any $\alpha\in\Xi_p$, where $\Xi_p$ is defined in Theorem (\ref{multipliers}), the C*-algebra $C^\ast\left(\Qadicsq{p},\Psi_\alpha\right)$, where the multiplier $\Psi_\alpha$ was defined in Theorem (\ref{multipliers}), is denoted by $\algebra_\alpha$. In this section, we will work with $p$ a prime number. Last, we also recall that by Notation (\ref{qtorus-not}), the rotation algebra for the rotation of angle $\theta\in[0,1)$ is denoted by $A_\theta$, while its canonical unitary generators are denoted by $U_\theta$ and $V_\theta$, so that $U_\theta V_\theta = e^{2i\pi\theta} V_\theta U_\theta$.

Theorem (\ref{K-theory-thm}) describes the $K_0$ groups of noncommutative solenoids, and, among other conclusions, state that there always exists a unique trace on the $K_0$ of any noncommutative solenoid, lifted from any tracial state on the C*-algebra itself. With this in mind, we state:

\begin{proposition}\label{proj-prop}
Let $p$ be a prime number, and fix $\alpha\in \Xi_p,$ with $\alpha_0\not\in\Q$. Let $\gamma = z + q\alpha_N$ for some $z,q \in \Z$ and $N\in\N$, with $\gamma>0$. Then there is a left projective module over $\algebra_\alpha$ whose $K_0$ class has trace $\gamma$, or equivalently, whose $K_0$ class is given by $\left(z,\frac{q}{p^k}\right) \in \Z\times\Qadic{p}$.
\end{proposition}
\begin{proof} By Remark (\ref{trace-range-rmk}), $\gamma$ is the image of some class in $K_0(\algebra_\alpha)$ for the trace on this group. Now, since $\alpha_{N+1} = p\alpha_N + j$ for some $j\in \Z$ by definition of $\Xi_\alpha$, we may as well assume $N$ is even. As $K_0(\algebra_\alpha)$ is the inductive limit of $K_0(A_{\alpha_{k}})_{k\in2\N}$ by Theorem (\ref{inductive-limit-thm}),  $\gamma$ is the trace of an element of $K_0(A_{\alpha_N}))$, where $A_{\alpha_N}$ is identified as a subalgebra of $\algebra_{\alpha}$ (again using Theorem (\ref{inductive-limit-thm}).  By \cite{Rieffel81}, there is a projection $P_\gamma$ in $A_{\alpha_N}$ whose $K_0$ class has trace $\gamma$, and it is then easy to check that the left projective module $P\algebra_\alpha$ over $\algebra_\alpha$ fulfills our proposition.
\end{proof}

So, for example, with the notations of the proof of Proposition (\ref{proj-prop}), if $P_{\gamma}$ is a projection in $A_{\alpha_n}\subset \algebra_\alpha$ with trace $\gamma\in (0,1)$, one can construct the equivalence bimodule 
\begin{equation*}
\algebra_\alpha \quad - \quad \algebra_\alpha P_{\gamma} \quad - \quad P_{\gamma}\algebra_\alpha P_{\gamma}\text{.}
\end{equation*}
From this realization, not much about the structure of $P_{\gamma}\algebra_\alpha P_{\gamma}$ can be seen, although it is possible to write this $C^{\ast}$-algebra as a direct limit of rotation algebras. Let us now discuss this matter.

Suppose we have two directed sequences of $C^{\ast}$-algebras: 
\begin{equation*}
\begin{CD}
A_{0} @>\varphi_0>> A_{1} @>\varphi_1>> A_{2} @>\varphi_2>> \cdots
\end{CD}
\end{equation*}
and 
\begin{equation*}
\begin{CD}
B_{0} @>\psi_0>> B_{1} @>\psi_1>> B_{2} @>\psi_2>> \cdots
\end{CD}
\end{equation*}
Suppose further that for each $n\in\N$ there is an equivalence bimodule $X_n$ between $A_{n}$ and $B_{n}$
$$ A_n - X_n - B_n,$$ and that the $(X_n)_{n\in\N}$ form a directed system, in the following sense: there exists a direct system of module monomorphisms  
\begin{equation*}
\begin{CD}
X_{0} @>i_0>> X_{1} @>i_1>> X_{2} @>i_2>> \cdots
\end{CD}
\end{equation*}
satisfying, for all $f,g \in X_n$ and $b\in B_n$:
\begin{align*}
\langle i_n(f), i_n(g) \rangle_{B_{n+1}}&=\psi_n(\langle f, g \rangle_{B_{n}})
\intertext{and}
i_n(f\cdot b) &= i_n(f)\cdot \psi_n(b)\text{,}
\end{align*}
with analogous but symmetric equalities holding for the $X_n$ viewed as left-$A_n$ modules. 

Now let $\mathcal{A}$ be the direct limit of $(A_n)_{n\in\N}$, $\mathcal{B}$ be the direct limit of $(B_n)_{n\in\N}$ and $\mathcal{X}$ be the direct limit of $(X_n)_{n\in\N}$ (completed in the natural $C^{\ast}$-module norm). Then $\mathcal{X}$ is an $\mathcal{A}-\mathcal{B}$ bimodule.  
If one further assumes that the algebra of adjointable operators on ${\mathcal X}$ viewed as a $\mathcal{A}-\mathcal{B}$ bimodule, ${\mathcal L}({\mathcal{X}}),$  can be obtained via an appropriate limiting process from the sequence of adjointable operators $\{{\mathcal L}(X_n)\}_{n=1}^{\infty}$  ( where each $X_n$ is a $A_n-B_n$ bimodule), 
then in addition one has that $\mathcal{X}$ is a strong Morita equivalence bimodule between  $\mathcal{A}$ and $\mathcal{B}.$

So suppose that $\gamma\in (0,1)$ is as in the statement of Proposition (\ref{proj-prop}), for some $\alpha\in \Xi_p$ not equal to zero, and suppose that we know that there is a positive integer $N$ and a projection $P_\gamma$ in $A_{\alpha_N}$ whose $K_0$ class has trace $\gamma.$  Again, without loss of generality, we assume that $N$ is even.  Then setting  
$$A_n=A_{\alpha_{N+2n}},\;X_n=A_{\alpha_{N+2n}}P_\gamma,\;\text{and}\;\; B_n= P_\gamma A_{\alpha_{N+2n}}P_\gamma,$$ 
all of the conditions in the above paragraphs hold {\it a priori}, since 
$\algebra_\alpha$ is a direct limit of the $A_{\alpha_{N+2n}},$ so that certainly ${\mathcal B}=P_\gamma\algebra_\alpha P_\gamma$ is a direct limit of the $P_\gamma A_{\alpha_{N+2n}}P_\gamma,$ and ${\mathcal X}= \algebra_\alpha P_{\gamma}$ can be expressed as a direct limit of the $X_n=A_{\alpha_{N+2n}}P_\gamma,$ again by construction, with the desired conditions on the adjointable operators satisfied by construction.  

It would be interesting to see how far this set-up could be extended to more general directed systems of Morita equivalence bimodules over directed systems of $C^{\ast}$-algebras, but we leave this project to a future endeavor.

We discuss very simple examples, to show how the directed system of bimodules is constructed.
\begin{example}
Fix an irrational $\alpha_0\in [0,1)$, let $p=2$, and consider $\alpha\in \Xi_2$ given by 
$$\alpha\; =\;(\alpha_0, \alpha_1=\frac{\alpha_0}{2}, \alpha_2=\frac{\alpha_0}{4},\cdots, \alpha_n=\frac{\alpha_n}{2^n},\cdots, ),$$
Consider $P_{\alpha_0}\in A_{\alpha_0}\subset A_{\alpha_1}$ a projection of trace $\alpha_0=2\alpha_1.$
The bimodule $$A_{\alpha_0}- A_{\alpha_0}\cdot P_{\alpha_0} - P_{\alpha_0}A_{\alpha_0}P_{\alpha_0}$$
is equivalent to Rieffel's bimodule 
$$A_{\alpha_0}-\overline{C_c(\mathbb R)}- A_{\frac{1}{\alpha_0}}=B_0.$$
Let $\beta_0=\frac{1}{\alpha_0}.$  Rieffel's theory, specifically Theorem 1.1 of \cite{Rieffel83}, again shows there is a bimodule
$$A_{\alpha_2}- A_{\alpha_2}\cdot P_{\alpha_0} - P_{\alpha_0}A_{\alpha_2}P_{\alpha_0}$$
is the same as 
$$A_{\alpha_2}- A_{\alpha_2}\cdot P_{4\alpha_2} - P_{4\alpha_2}A_{\alpha_2}P_{4\alpha_2}$$
which is equivalent to Rieffel's bimodule
$$A_{\alpha_2}-\overline{C_c(\mathbb R\times F_4})- C(\mathbb T\times F_4)\rtimes_{\tau_1}\mathbb Z=B_1,$$
where $F_4=\mathbb Z/4\mathbb Z,$ and the action of $\mathbb Z$ on $\mathbb T\times F_4$ is given by multiples of $(\frac{\beta_2}{4},[1]_{F_4}),$ for $\beta_2=\frac{1}{\alpha_2},$ i.e. multiples of 
$(\frac{1}{\alpha_0},[1]_{F_4}]),$ i.e. multiples of  $(\beta_0,[1]_{F_4}).$

At the $n^{th}$ stage, using Theorem 1.1 of \cite{Rieffel83} again, we see that  
 $$A_{\alpha_{2n}}- A_{\alpha_{2n}}\cdot P_{\alpha_0} - P_{\alpha_0}A_{\alpha_{2n}}P_{\alpha_0}$$
is the same as 
$$A_{\alpha_{2n}}-A_{\alpha_{2n}}\cdot P_{2^n\alpha_n} - P_{2^n\alpha_n}A_{\alpha_n}P_{2^n\alpha_n}$$
which is equivalent to 
$$A_{\alpha_n}-\overline{C_c(\mathbb R\times F_{4^n}})- C(\mathbb T\times F_{4^n})\rtimes_{\tau_n}\mathbb Z=B_n,$$
where the action of $\mathbb Z$ on $\mathbb T\times F_{4^n}$ is given by multiples of $(\frac{\beta_{2n}}{4^n},[1]_{F_{4^n}}),$ for $\beta_{2n}=\frac{1}{\alpha_{2n}}=\frac{4^n}{\alpha_0},$ i.e. multiples of 
$(\frac{1}{\alpha_0},[1]_{F_{4^n}}),$ i.e. multiples of  $(\beta_0,[1]_{F_{4^n}}),$ for $F_{4^n}=\mathbb Z/4^n\mathbb Z.$

From calculating the embeddings, we see that for $\alpha=(\alpha_0,\frac{\alpha_0}{2},\cdots, \frac{\alpha}{2^n},\cdots )\in \Xi_2,$ we have that 
$${\mathcal A}^{\mathcal S}_{\alpha}$$ is strongly Morita equivalent to a direct limit ${\mathcal B}$ of the $B_n.$ The structure of ${\mathcal B}$ is not clear in this description, although each $B_n$ is seen to be a variant of a rotation algebra. As expected, one calculates 
$$tr(K_0({\mathcal A}^{\mathcal S}_{\alpha}))=\alpha_0\cdot tr(K_0({\mathcal B})).$$
\end{example}

\section{Forming projective modules over noncommutative solenoids using $p$-adic fields }

Under certain conditions, one can construct equivalence bimodules for $\algebra_\alpha$ ($\alpha\in\Xi_p$,$p$ prime) by using a construction of M. Rieffel \cite{Rieffel88}.  The idea is to first embed $\Gamma = \Qadicsq{p}$ as a co-compact `lattice' in a larger group $M$, and the quotient group $\bigslant{M}{\Gamma}$ will be exactly the solenoid $\solenoid_p$.  We thank Jerry Kaminker and Jack Spielberg for telling us about  this trick. 

We start with a brief description of the field of $p$-adic numbers, with $p$ prime. Algebraically, the field $\Q_p$ is the field of fraction of the ring of $p$-adic integers $\Z_p$ --- we introduce $\Z_p$ as a group, though there is a natural multiplication on $\Z_p$ turning it into a ring. A more analytic approach is to consider $\Q_p$ as the completion of the field $Q$ for the $p$-adic metric $d_p$, defined by $d_p(r,r') = |r-r'|_p$ for any $r,r'\in\Q$, where $|\cdot|_p$ is the $p$-adic norm defined by:
\begin{equation*}
|r| = \begin{cases}
p^{-n} \text{ if $r\not=0$ and where $r = p^n \frac{a}{b}$ with $a,b$ are both relatively prime with $p$},\\
0 \text{ if $r=0$.}
\end{cases}
\end{equation*}
If we endow $\Q$ with the metric $d_p$, then series of the form:
\begin{equation*}
\sum_{j=k}^\infty a_j p^j
\end{equation*}
will converge, for any $k \in \Z$ and $a_j \in \{0,\ldots,p-1\}$ for all $j=k,\ldots$. This is the $p$-adic expansion of a $p$-adic number. One may easily check that addition and multiplication on $\Q$ are uniformly continuous for $d_p$ and thus extend uniquely to $\Q_p$ to give it the structure of a field. Moreover, one may check that the group $\Z_p$ of $p$-adic integer defined in Section 3 embeds in $\Q_p$ as the group of $p$-adic numbers of the form $\sum_{j=0}^\infty a_j p^j$ with $a_j \in \{0,\ldots,p-1\}$ for all $j\in\N$. Now, with this embedding, one could also check that $\Z_p$ is indeed a subring of $\Q_p$ whose field of fractions is $\Q_p$ (i.e. $\Q_p$ is the smallest field containing $\Z_p$ as a subring) and thus, both constructions described in this section agree. Last, the quotient of the (additive) group $\Q_p$ by its subgroup $\Z_p$ is the Pr\"ufer $p$-group $\Z(p^\infty) = \{ z \in \T : \exists n \in \N \quad z^{\left(p^n\right)} = 1\}$.

\subsection{Embedding $\mathbb Z(\frac{1}{p})$ as a lattice in a self-dual group }

Since $\Q_p$ is a metric completion of $\Q$ and $\Qadic{p}$ is a subgroup of $\Q$, we shall identify, in this section, $\Qadic{p}$ as a subgroup of $\Q_p$ with no further mention. We now define a few group homomorphisms to construct a short exact sequence at the core of our construction.

Let $\omega : \R\rightarrow \solenoid_p$ be the standard ``winding line'' defined for any $t\in \R$ by:
\begin{equation*}
\omega(t)=\left(e^{2\pi it},e^{2\pi i\frac{t}{p}}, e^{2\pi i\frac{t}{p^2}},\cdots, e^{2\pi i\frac{t}{p^n}}, \cdots \right)\text{.}
\end{equation*}

Let $\gamma \in \Q_p$ and write $\gamma = \sum_{j=k}^\infty a_jp^j$ for a (unique) family $(a_j)_{j=k,\ldots}$ of elements in $\{0,\ldots,p-1\}$. We define the sequence $\zeta(\gamma)$ by setting for all $j\in \N$:
\begin{equation*}
\zeta_j(\gamma)=e^{2\pi i\left(\sum_{m=k}^{j}\frac{a_{m}}{p^{j-m+k}}\right)}
\end{equation*}
with the convention that $\sum_j^k \cdots$ is zero if $k<j$. 

We thus may define the map
\begin{equation*}
\Pi:\begin{cases}
\Q_p\times \R &\longrightarrow \solenoid_p\\
\gamma &\longmapsto\Pi(\gamma, t)= \zeta_j(\gamma)\cdot \omega(t)\text{.}
\end{cases}
\end{equation*}

If we set
\begin{equation*}
\iota:\begin{cases}
\Qadic{p}&\longrightarrow  \Q_p\times \R\\
r &\longmapsto \iota(r)=(r,-r)\text{,}
\end{cases}
\end{equation*}
then one checks that the following is an exact sequence:\
\begin{equation*}
\begin{CD}
 1 @>>> \Qadic{p} @>\iota>> \Q_p\times \R @>\Pi>> \solenoid_p @>>> 1
\end{CD}
\end{equation*}

It follows that there is an exact sequence 
\begin{equation*}
\begin{CD}
 1 @>>> \Qadicsq{p}  @>>> [\Q_p\times \R]\times [\Q_p\times \R] @>>> \solenoid_p\times\solenoid_p @>>> 1\text{.}
\end{CD}
\end{equation*}

Indeed, we will show later that it is possible to perturb the embeddings of the different terms in $\Qadicsq{p}$ by elements of $\Q_p\setminus\{0\}$ and $\R\setminus\{0\}$ to obtain a family of different embeddings of $\Qadicsq{p}$ into $[\Q_p\times \R]^2$.

We now observe that $M = \Q_p\times \R$ is self-dual. We shall use the following standard notation:
\begin{notation}
The Pontryagin dual of a locally compact group $G$ is denoted by $\dual{G}$. The dual pairing between a group and its dual is denoted by $\left<\cdot,\cdot\right>: G\times\dual{G}\rightarrow\T$.
\end{notation}

Let us show that $M\cong\widehat{M}$. To every $x\in \Q_p$, we can associate the character
\begin{equation*}
\chi_x : q\in \Q_p \mapsto e^{2i\pi i \{x\cdot q\} }
\end{equation*}
where $\{x\cdot q\}_p$ is the fractional part of the product $x\cdot q$ in $\Q_p,$ i.e. it is the sum of the terms involving the negative powers of $p$  in the $p$-adic expansion of $x\cdot q$. The map $x\in\Q_p \mapsto \chi_x \in \dual{Q_p}$ is an isomorphism of topological group. Similarly, every character of $\R$ is of the form $\chi_r:t \in \R \mapsto e^{2i\pi rt}$ for some $r\in \R$. Therefore every character of $M$ is given by
\begin{equation*}
\chi_{(x,r)}: (q,t)\in\Q_p\times\R  \longmapsto \chi_x(q)\chi_r(t)
\end{equation*}
for some $(x,r)\in\Q_p\times\R$ (see \cite{HR70}) for further details on characters of specific locally compact abelian groups). It is possible to check that the map $(x,r)\mapsto \chi_{(x,r)}$ is a group isomorphism between $M$ and $\dual{M},$ so that $M=\Q_p\times \R$ is indeed self-dual.

\subsection{The Heisenberg representation and the Heisenberg equivalence bimodule of Rieffel}

In this section, we write $\Gamma=\Qadicsq{p}$ where $p$ is some prime number, and now let $M=[\Q_p\times \R]$.  We have shown in the previous section that $M$ is self-dual, since both $\Q_p$ and $\R$ are self-dual.  Now suppose there is an embedding $\iota:\Gamma\rightarrow M\times \dual{M}.$  Let the image $\iota(\Gamma)$ be denoted by $D$. In the case we are considering, $D$ is a discrete co-compact subgroup of $M\times \dual{M}$.  Following the method of M. Rieffel \cite{Rieffel88}, the {\bf Heisenberg multiplier} $\eta:(M\times \dual{M})\times (M\times \dual{M})\to \mathbb T$  is defined by:
$$\eta((m,s),(n,t))=\langle m, t\rangle,\; (m,s), (n,t)\in   M\times \dual{M}.$$  (We note we use the Greek letter `$\eta$' rather than Rieffel's `$\beta$', because we have used `$\beta$' elsewhere.
Following Rieffel, the  {\bf symmetrized version} of $\eta$ is denoted by the letter $\rho,$ and is the multiplier defined by:
\begin{equation*}
\rho((m,s),(n,t))=\eta((m,s),(n,t))\overline{\eta((n,t),(m,s))},\; (m,s), (n,t)\in   M\times \dual{M}\text{.} 
\end{equation*}

M. Rieffel \cite{Rieffel88} has shown that $C_C(M)$ can be given the structure of a left  $C^{\ast}(D,\eta)$ module, as follows.
One first constructs an $\eta$-representation of $M\times \dual{M}$ on $L^2(M),$ defined as $\pi,$ where 
\begin{equation*}
\pi_{(m,s)}(f)(n)=\left<n,s\right> f(n+m),  (m,s)\in M\times \dual{M},\;n\in M\text{.} 
\end{equation*}
When the representation $\pi$ is restricted to $D,$ we still have a projective $\eta$-representation of $D,$ on $L^2(M),$ and its integrated form gives $C_C(M)$ the structure of a left $C^{\ast}(D,\eta)$ module, i.e. for $\Phi\in C_C(D,\eta),\; f\in C_C(M),$
\begin{equation*}
\begin{split}
\pi(\Phi)\cdot f (n)&=\sum_{(d,\chi)\in D}\Phi((d,\chi))\pi_{(d,\chi)}(f)(n)\\
&=\sum_{(d,\chi)\in D}\Phi((d,\chi))\left<n,\chi\right>f(n+d)\text{.}
\end{split}
\end{equation*}

There is also a $C_C(D,\eta)$ valued inner product defined on $C_C(M)$ given by:
$$\langle f, g\rangle_{C_C(D,\eta)}=\int_M f(n)\overline{\pi_{(d,\chi)}(g)(n)}dn=\;\int_M f(n)\overline{\left<n,\chi\right>g(n+d)}dn.$$

Moreover, Rieffel has shown that setting 
\begin{equation*}
D^{\perp}=\{(n,t)\in M\times \dual{M}: \forall (m,s)\in D\quad \rho((m,s),(n,t))=1\}\text{,} 
\end{equation*}
$C_C(M)$ has the structure of a right $C^{\ast}(D^{\perp},\overline{\eta})$ module.
Here the right module structure is given for all $f\in C_c(M)$, $\Omega\in C_c(D^\perp)$ and $n \in M$ by:
\begin{equation*}
f\cdot \Omega (n)=\sum_{(c,\xi)\in D^{\perp}}\pi_{(c,\xi)}^*(f)(n)\Omega(c,\xi),
\end{equation*}
and the $C_C(D^{\perp},\overline{\eta})$-valued inner product is given by 
\begin{equation*}
\langle f, g\rangle_{C_C(D^{\perp},\overline{\eta})}(c,\xi)=\int_M\overline{f(n)}\pi_{(c,\xi)}(g)(n)dn=\int_M\overline{f(n)}\left<n,\xi\right>g(n+c)dn\text{,}
\end{equation*}
where $f, g \in C_C(M),\; \Omega\in C_C(D^{\perp},\overline{\eta}),$ and $(c,\xi)\in D^{\perp}.$

Moreover, Rieffel shows in \cite[Theorem 2.12]{Rieffel88} that $C^{\ast}(D,\eta)$ and  $C^{\ast}(D^{\perp},\overline{\eta})$ are strongly Morita equivalent, with the equivalence bimodule being the completion of $C_C(M)$ in the norm defined by the above inner products.

In order to construct explicit bimodules, we first define the multiplier $\eta$ more precisely, and then discuss different embeddings of $\Qadicsq{p}$ into $M\times \dual{M}.$

In the case examined here, the Heisenberg multiplier $\eta: [\Q_p\times \R]^2\times [\Q_p\times \R]^2\to \mathbb T$ is given by:
\begin{definition} The Heisenberg multiplier $\eta: [\Q_p\times \R]^2\times [\Q_p\times \R]^2\to \T$ is defined by 
\begin{equation*}
\eta[((q_1,r_1),(q_2,r_2)),((q_3,r_3),(q_4,r_4))]= e^{2\pi i r_1r_4}e^{2\pi i \{q_1q_4\}_p}\text{,}
\end{equation*}
where $\{q_1q_4\}_p$ is the fractional part of the product $q_1\cdot q_4,$ i.e. the sum of the terms involving the negative powers of $p$  in the $p$-adic expansion of $q_1q_4.$
\end{definition}
 
The following embeddings of $\Qadicsq{p}$ in $[\Q_p\times\R]^2$ will prove interesting:
\begin{definition}
For $\theta\in \R,\theta\not=0,$ we define $\iota_{\theta}:\Qadicsq{p}\rightarrow [\Q_p\times \R]^2$ by 
\begin{equation*}
\iota_{\theta}(r_1, r_2)=[(r_1, \theta\cdot r_1), (r_2, r_2)]\text{.}
\end{equation*}
\end{definition}

We examine the structure of the multiplier $\eta$ more precisely and then discuss different embeddings of $\Qadicsq{p}$ into $[\Q_p\times \R]^2$ and their influence on the different equivalence bimodules they allow us to construct.

We start by observing that for $r_1,r_2,r_3,r_4\in \Qadic{p}$:
\begin{equation*}
\begin{split}
\eta(\iota_{\theta}(r_1,r_2)),\iota_{\theta}(r_3,r_4))\ &= e^{2\pi i \{r_1 r_4\}_p}e^{2\pi i \theta r_1r_4}\\
&=e^{2\pi i r_1r_4}e^{2\pi i \theta r_1r_4}=e^{2\pi i (\theta+1)r_1r_4}\text{.}
\end{split}
\end{equation*}
(Here we used the fact that for $r_i, r_j\in \mathbb Z(\frac{1}{p}),\;\{ r_i r_j\}_p\equiv r_ir_j$ modulo $\mathbb Z.$)

One checks that  setting $D_{\theta}=\iota_{\theta}\left(\Qadicsq{p}\right)$, the $C^{\ast}$-algebra $C^{\ast}(D_{\theta},\eta)$ is exactly *-isomorphic to the noncommutative solenoid $\algebra_\alpha$ for 
\begin{equation*}
\alpha=\left(\theta+1, \frac{\theta+1}{p},\cdots, \frac{\theta+1}{p^n},\cdots \right) = \left(\frac{\theta+1}{p^n}\right)_{n\in\N} \text{.}
\end{equation*}

For this particular embedding of $\Qadicsq{p}$ as the discrete subgroup $D$ inside $M\times \dual{M},$ we calculate that  
\begin{equation*}
D_{\theta}^{\perp}=\left\{\left(r_1,-\frac{r_1}{\theta}\right),(r_2,-r_2):\;r_1,\;r_2\in \Qadic{p} \right\}\text{.}
\end{equation*}
Moreover,
\begin{equation*}
\overline{\eta}\left(\left[\left(r_1,-\frac{r_1}{\theta}\right),\left(r_2,-r_2\right)\right],\left[\left(r_3,-\frac{r_3}{\theta}\right),\left(r_4,-r_4\right)\right]\right)=e^{-2\pi i(\frac{1}{\theta}+1)r_1r_4}\text{.}
\end{equation*}
It is evident that $C^{\ast}(D_{\theta}^{\perp},\eta)$ is also a non-commutative solenoid $\algebra_\beta$ where $\beta=\left(1-\frac{\theta+1}{p^n\theta}\right)_{n\in\N}$.

Note that for 
\begin{equation*}
\alpha=\left(\theta+1, \frac{\theta+1}{p},\cdots, \frac{\theta+1}{p^n},\cdots \right),
\end{equation*}
and 
\begin{equation*}
\beta=\left(1-\frac{\theta+1}{p^n\theta}\right)_{n\in\N},
\end{equation*}
we have
\begin{equation*}
\theta \cdot \tau\left(K_0\left(\algebra_\alpha\right)\right) =\tau\left(K_0\left(\algebra_\beta\right)\right)
\end{equation*}
with the notations of Theorem (\ref{K-theory-thm}). Thus in this case we do see the desired relationship mentioned in Section 4: the range of the trace on the $K_0$ groups of the two $C^{\ast}$-algebras are related via multiplication by a positive constant.

We can now generalize our construction above as follows. 
\begin{definition}\label{iota-def}
For any $x\in\Q_p\setminus\{0\}$, and any $\theta\in\R\setminus\{0\},$ there is an embedding 
\begin{equation*}
\iota_{x,\theta}:\Qadicsq{p}\to [\Q_p\times\R]^2
\end{equation*}
defined for all $r_1,r_2\in \Qadic{p}$ by 
\begin{equation*}
\iota_{x,\theta}(r_1,r_2)=[(x\cdot r_1,\theta\cdot r_1), (r_2, r_2)]\text{.}
\end{equation*}
\end{definition}

Then, we shall prove that for all $\alpha\in\Xi_p$ there exists $x\in\Q_p\setminus\{0\}$ and $\theta\in\R\setminus\{0\}$ such that, by setting
\begin{equation*}
D_{x,\theta}=\iota_{x,\theta}\left(\Qadicsq{p}\right)
\end{equation*}
the twisted group C*-algebra $C^\ast(D,\eta)$ is *-isomorphic to $\algebra_\alpha$. 

As a first step, we prove:
\begin{lemma}\label{a-lemma}
Let $p$ be prime, and let $M=\Q_p\times \R.$  Let $(x,\theta)\in [{\Q}_p\setminus\{0\}]\times [\R\setminus\{0\}],$ and define $\iota_{x,\theta}: \Qadicsq{p}\to [\Q_p\times\R]^2\cong M\times \dual{M}$ by:
\begin{equation*}
\iota_{x,\theta}(r_1,r_2)=[(x\cdot r_1,\theta\cdot r_1), (r_2, r_2)]\text{ for all $r_1,r_2\in\Qadic{p}$.}
\end{equation*}
Let $\eta$ denote the Heisenberg cocycle defined on $[M\times \dual{M}]^2$ and let
\begin{equation*}
D=\iota_{x,\theta}\left(\Qadicsq{p}\right)\text{.}
\end{equation*}
Then 
\begin{equation*}
D_{x,\theta}^{\perp}=\left\{\left[(t_1,-t_1), \left(x^{-1}t_2,-\frac{t_2}{\theta}\right)\right]: t_1,t_2\in \Qadic{p} \right\}.
\end{equation*}
\end{lemma}
\begin{proof}
By definition,
\begin{equation*}
\begin{split}
D_{x,\theta}^{\perp}&=\left\{[(q_1,s_1), (q_2,s_2)]:\forall r_1,r_2\in\Qadic{p}\quad \rho([\iota_{x,\theta}(r_1,r_2)], [(q_1,s_1), (q_2,s_2)])=1\right\}\\
&=\left\{[(q_1,s_1), (q_2,s_2)]: \forall r_1,r_2\in\Qadic{p}\quad\rho([(x\cdot r_1,\theta\cdot r_1), (r_2, r_2)],  [(q_1,s_1), (q_2,s_2)])=1 \right\}\\
&=\left\{[(q_1,s_1), (q_2,s_2)]:\forall r_1,r_2\in\Qadic{p}\quad e^{2\pi i \theta r_1s_2}e^{2\pi i \{x\cdot r_1q_2\}_p}\overline{e^{2\pi i s_1r_2}e^{2\pi i \{q_1r_2\}_p}}=1
\right\}\text{.}
\end{split}
\end{equation*}
Now if $r_2=0,$ and $r_1=p^n,$ for any $n\in \Z,$ this implies 
\begin{equation*}
\forall n\in\Z\quad e^{2\pi i \theta p^n s_2}e^{2\pi i \{x\cdot p^n q_2\}_p}=1, 
\end{equation*}
so that if we choose $s_2=-\frac{t_2}{\theta}$ for some $t_2\in \Qadic{p}\subseteq \R,$ we need $q_2= x^{-1}t_2.$
Likewise, if we take $r_1=0,$ and $r_2=p^n,$ for any $n\in \Z,$ we need $(q_1,s_1)$ such that 
\begin{equation*}
\forall n \in\Z \quad e^{2\pi i s_1p^n}e^{2\pi i \{q_1 p^n\}_p}=1.
\end{equation*}
Again fixing $q_1=t_1\in \Qadic{p},$ this forces $s_1=-t_1.$
Thus  
\begin{equation*}
D_{x,\theta}^{\perp}=\left\{\left[(t_1,-t_1), \left(x^{-1}t_2,-\frac{t_2}{\theta}\right)\right]: t_1,\;t_2\in\Qadic{p}\right\}\text{,}
\end{equation*}
as we desired to show.
\end{proof}

One thus sees that the two $C^{\ast}$-algebras $C^{\ast}(D_{x,\theta},\eta)$ and $C^{\ast}(D_{x,\theta}^{\perp},\overline{\eta})$ are strongly Morita equivalent (but not isomorphic, in general), and also the proof of this lemma shows that $C^{\ast}(D_{x,\theta}^{\perp},\overline{\eta})$ is a noncommutative solenoid.

We can use Lemma (\ref{a-lemma}) to prove the following Theorem:

\begin{theorem}\label{thmSME}
Let $p$ be prime, and let $\alpha=(\alpha_i)_{i\in\N}\in \Xi_p,$ with $\alpha_0\in (0,1).$   Then there exists $(x,\theta)\in [{\Q}_p\setminus\{0\}]\times [\R\setminus\{0\}]$ with $C^{\ast}(D_{x,\theta},\eta)$  isomorphic to the noncommutative solenoid $\algebra_\alpha$, where $D_{x,\theta} = \iota_{x,\theta} \left(\Qadicsq{p}\right)$.
Moreover, the method of Rieffel produces an equivalence bimodule between  $\algebra_\alpha$ and another unital $C^{\ast}$-algebra ${\mathcal B},$ and ${\mathcal B}$ is itself isomorphic to a noncommutative solenoid.
\end{theorem}

\begin{proof}
By definition of $\Xi_p$, for all $j\in\N$ there exists $b_j \in \{0,\ldots,p-1\}$ such that $p\alpha_{j+1}=\alpha_j+b_j$.  We construct an element of the $p$-adic integers, $x=\sum_{j=0}^{\infty}b_jp^j\in\Z_p\subset \Q_p.$  Let $\theta=\alpha_0,$ and now consider for this specific $x$ and this specific $\theta$ the $C^{\ast}$-algebra $C^{\ast}(D_{x,\theta},\eta)$.  By Definition (\ref{iota-def}), $\iota_{x,\theta}(r_1,r_2)=[(x\cdot r_1,\theta\cdot r_1), (r_2, r_2)],$ for $r_1,r_2\in\Qadic{p}.$
Then 
\begin{equation*}
\begin{split}
\eta\left(\iota_{x,\theta}(r_1,r_2), \iota_{x,\theta}(r_3,r_4)\right)&=\eta\left(\left[\left(x\cdot r_1,\theta\cdot r_1\right), \left(r_2, r_2\right)\right],\left[\left(x\cdot r_3,\theta\cdot r_3\right), \left(r_4, r_4\right)\right]\right)\\
&=e^{2\pi i \theta r_1r_4}e^{2\pi i \{xr_1r_4\}_p},\;r_1, r_2, r_3, r_4\in \Qadic{p}\text{,}
\end{split}
\end{equation*}
and, setting $r_i=\frac{j_i}{p^{k_i}},\;1\leq i\leq 4,$ and setting $\theta=\alpha_0,$ we obtain 
\begin{equation*}
\eta\left(\iota_{x,\alpha_0}\left(\frac{j_1}{p^{k_1}},\frac{j_2}{p^{k_2}}\right), \iota_{x,\alpha_0}\left(\frac{j_3}{p^{k_3}},\frac{j_4}{p^{k_4}}\right)\right)
=e^{2\pi i \alpha_0 \frac{j_1j_4}{p^{k_1+k_4}}}e^{2\pi i \{x\frac{j_1j_4}{p^{k_1+k_4}}\}_p}
\end{equation*}
for all
\begin{equation*}
\frac{j_1}{p^{k_1}}, \frac{j_2}{p^{k_2}}, \frac{j_3}{p^{k_3}}, \frac{j_4}{p^{k_4}} \in\Qadic{p}\text{.}
\end{equation*}
We now note that the relation $p\alpha_{j+1}=\alpha_j+b_j,\; b_j\in\;\{0,1,\cdots, p-1\}$ allows us to prove inductively that 
\begin{equation*}
\forall n\geq 1\quad \alpha_n=\frac{\alpha_0+\sum_{j=0}^{n-1}b_jp^j}{p^n}.
\end{equation*}
By Theorem (\ref{multipliers}), the multiplier $\Psi_{\alpha}$ on $\Qadicsq{p}$ is defined by:
\begin{equation*}
\begin{split}
\Psi_{\alpha}\left(\left(\frac{j_1}{p^{k_1}},\frac{j_2}{p^{k_2}}\right),\left(\frac{j_3}{p^{k_3}},\frac{j_4}{p^{k_4}}\right)\right)&=e^{2\pi i(\alpha_{(k_1+k_4)}j_1j_4)}\\
&=e^{2\pi i\frac{\alpha_0j_1j_4}{p^{k_1+k_4}}}e^{2\pi i (\sum_{j=0}^{k_1+k_4-1}b_jp^jj_1j_4)/p^{k_1+k_2}}.
\end{split}
\end{equation*}
A $p$-adic calculation now shows that for $\frac{j_1}{p^{k_1}}$ and $\frac{j_4}{p^{k_4}}\in \Qadic{p}$  and $x=\sum_{j=0}^{\infty}b_jp^j\in\Z_p,$ we have 
$\{x\frac{j_1j_4}{p^{k_1+k_4}}\}_p=(\sum_{j=0}^{k_1+k_4-1}b_jp^j)\cdot \frac{j_1j_4}{p^{k_1+k_2}}$ modulo $\Z,$ so that 
\begin{equation*}
e^{2\pi i \{x\frac{j_1j_4}{p^{k_1+k_4}}\}_p}=e^{2\pi i (\sum_{j=0}^{k_1+k_4-1}b_jp^jj_1j_4)/p^{k_1+k_2}}.
\end{equation*}
We thus obtain 
\begin{equation*}
\eta(\iota_{x,\theta}(r_1,r_2), \iota_{x,\theta}(r_3,r_4))=\Psi_{\alpha}((r_1,r_2),(r_3,r_4))
\end{equation*}
for all $ r_1,r_2,r_3,r_4\in \Qadic{p}$,as desired.

To prove the final statement of the Theorem, we use Lemma \ref{a-lemma}.  We have shown $\algebra_\alpha$ is isomorphic to $C^{\ast}(D_{x,\theta},\eta),$ and the discussion prior to the statement of  Lemma \ref{a-lemma} shows that $C^{\ast}(D_{x,\theta},\eta)$ is strongly Morita equivalent to $C^{\ast}(D_{x,\theta}^{\perp},\overline{\eta})={\mathcal B}.$ But the proof of Lemma \ref{a-lemma} gives that $D_{x,\theta}^{\perp}$ is isomorphic to $\Qadicsq{p},$ so that  $C^{\ast}(D_{x,\theta}^{\perp},\overline{\eta})={\mathcal B}$ is a noncommutative solenoid, as we desired to show.
\end{proof}
\begin{remark} It remains an open question to give necessary and sufficient conditions under which two noncommutative solenoids $\algebra_\alpha$ and $\algebra_\beta$ would be strongly Morita equivalent, although it is evident that a necessary that the range of the trace on $K_0$ of one of the $C^{\ast}$-algebras should be a constant multiple of the range of the trace on the $K_0$ group of the other. By changing the value of $\theta$ to be $\alpha_0+j,\; j\in \Z,$ and adjusting the value of $x\in \Q_p$ accordingly, one can use the method of Theorem \ref{thmSME} to construct a variety of embeddings $\iota_{x,\theta}$ of $\Qadicsq{p}$ into $[\Q_p\times \R]^2$ that provide lattices $D_{x,\theta}$ such that $C^{\ast}(D_{x,\theta},\eta)$ and $\algebra_\alpha$ are $\ast$-isomorphic, but such that the strongly Morita equivalent solenoids $C^{\ast}(D_{x,\theta}^{\perp},\overline{\eta})$ vary in structure.  This might lend some insight into classifying the noncommutative solenoids up to strong Morita equivalence, as might a study between the relationship between the two different methods of building equivalence modules described in Sections 4 and 5.
\end{remark}
\begin{remark}  In the case where the lattice $\Z^{2n}$ embeds into $\R^n \times \dual{ \R^n},$ F. Luef has used the Heisenberg equivalence bimodule construction of Rieffel to construct different families of Gabor frames in modulation spaces of $L^2(\R^n)$ for modulation and translation by $\Z^n$ (\cite{Luef09}, \cite{Luef11}). It is of interest to see how far this analogy can be taken when studying modulation and translation operators of $\Qadic{p}$ acting on $L^2(\Q_p\times \R),$ and we are working on this problem at present.
\end{remark}


\begin{thebibliography}{10}

\bibitem{BK76} {L.} {B}aggett and {A.} {K}leppner, 
\emph{Multiplier representations of abelian groups},  
J. Functional Analysis \textbf{14} (1973), 299–-324. 
 

\bibitem{Connes80}
{A}. {C}onnes, \emph{{C*}--alg{\`e}bres et g{\'e}om{\'e}trie differentielle},
  {C}. {R}. de l'academie des Sciences de Paris (1980), no.~series A-B, 290.

\bibitem{ER95} {S}. {E}chterhoff and {J}. {R}osenberg, \emph{Fine structure of the Mackey machine for actions of abelian groups 
with constant Mackey obstruction},
 Pacific J. Math. \textbf{170} (1995), 17-–52. 

\bibitem{Elliott93b}
{G.} {E}lliott and {D}. {E}vans, \emph{Structure of the irrational rotation
  {$C^\ast$}-algebras}, Annals of Mathematics \textbf{138} (1993), 477--501.


\bibitem{Fuchs70}
{L}. {F}uchs, \emph{Infinite Abelian Groups, Volume I}, Academic Press, New York and London, 1970.


\bibitem{HR70}
{E}. {H}ewitt and {K}. {R}oss, \emph{Abstract Harmonic Analysis, Volume II}, Springer-Verlag Berlin, 1970.


\bibitem{Hoegh-Krohn81}
R.~Hoegh-Krohn, M.~B. Landstad, and E.~Stormer, \emph{Compact ergodic groups of
  automorphisms}, {A}nnals of {M}athematics \textbf{114} (1981), 75--86.

\bibitem{Kleppner65}
A.~Kleppner, \emph{Multipliers on {A}belian groups}, Mathematishen Annalen
  \textbf{158} (1965), 11--34.

\bibitem{Latremoliere05}
{F}. {L}atr{\'e}moli{\`e}re, \emph{Approximation of the quantum tori by finite
  quantum tori for the quantum gromov-hausdorff distance}, Journal of Funct.
  Anal. \textbf{223} (2005), 365--395, math.OA/0310214.

\bibitem{Latremoliere11c}
{F}. {L}atr{\'e}moli{\`e}re and {J}. {P}acker, \emph{Noncommutative solenoids},
  Submitted (2011), 30 pages, ArXiv: 1110.6227.

\bibitem{Luef09}
F.~{L}uef, \emph{Projective modules over noncommutative tori and multi-window
  Gabor frames for modulation spaces}, J. Funct. Anal. \textbf{257} (2009),
  1921--1946.

\bibitem{Luef11}
\bysame, \emph{Projections in noncommutative tori and Gabor frames}, Proc.
  Amer. Math. Soc. \textbf{139} (2011), 571--582.

\bibitem{Packer92}
{J}. {P}acker and {I}. {R}aeburn, \emph{On the structure of twisted group
  {$C^\ast$-}algebras}, Trans. Amer. Math. Soc. \textbf{334} (1992), no.~2,
  685--718.

\bibitem{Rieffel81}
M.~A. {R}ieffel, \emph{{C*}-algebras associated with irrational rotations},
  Pacific Journal of Mathematics \textbf{93} (1981), 415--429.

\bibitem{Rieffel83}
\bysame, \emph{The cancellation theorem for the projective modules over
  irrational rotation {$C^\ast$}-algebras}, Proc. London Math. Soc. \textbf{47}
  (1983), 285--302.

\bibitem{Rieffel88}
\bysame, \emph{Projective modules over higher-dimensional non-commutative
  tori}, {C}an. {J}. {M}ath. \textbf{XL} (1988), no.~2, 257--338.

\bibitem{Rieffel90}
\bysame, \emph{Non-commutative tori --- a case study of non-commutative
  differentiable manifolds}, Contemporary Math \textbf{105} (1990), 191--211.

\bibitem{Rieffel98a}
\bysame, \emph{Metrics on states from actions of compact groups}, Documenta
  Mathematica \textbf{3} (1998), 215--229, math.OA/9807084.

\bibitem{Rieffel00}
\bysame, \emph{{G}romov-{H}ausdorff distance for quantum metric spaces}, Mem.
  Amer. Math. Soc. \textbf{168} (March 2004), no.~796, math.OA/0011063.

\bibitem{Zeller-Meier68}
{G}. {Z}eller{-}{M}eier, \emph{Produits crois{\'e}s {d'u}ne {C*-}alg{\`e}bre
  par un groupe {d' A}utomorphismes}, {J}. {M}ath. pures et appl. \textbf{47}
  (1968), no.~2, 101--239.

\end{thebibliography}

\providecommand{\bysame}{\leavevmode\hbox to3em{\hrulefill}\thinspace}
\providecommand{\MR}{\relax\ifhmode\unskip\space\fi MR }
\providecommand{\MRhref}[2]{%
  \href{http://www.ams.org/mathscinet-getitem?mr=#1}{#2}
}
\providecommand{\href}[2]{#2}

\vfill

\end{document}